\documentclass[matann2]{svjour}

\usepackage{amsmath}
\usepackage{amsfonts}
\usepackage{latexsym}
\usepackage{mathrsfs}
\usepackage{times}
\usepackage{array}
\usepackage{amscd}
\usepackage{a4}
\usepackage[mathcal]{euscript}
\usepackage{amssymb}
\usepackage{epsf, epic, eepic} 
\usepackage{pstricks}
\usepackage{amsrefs} 
\usepackage{theorem}
\usepackage{graphicx}

\input xy
\xyoption{all}

\input xy 
\xyoption{all} 
\newcommand{\bN}{\mathbb{N}} 
\newcommand{\bZ}{\mathbb{Z}}

\newcommand{\bB}{\mathbb{B}} 
 
\newcommand{\bH}{\mathbb{H}} 

\newcommand{\cC}{\mathcal{C}} 
 
\newcommand{\cF}{\mathcal{F}} 
\newcommand{\cG}{\mathcal{G}} 
\newcommand{\cH}{\mathcal{H}} 
\newcommand{\cK}{\mathcal{K}}

\newcommand{\cS}{\mathcal{S}}

\newcommand{\ra}{\rightarrow}

\newcommand{\gr}[1]{| {#1}|}

\newcommand{\cov}[1]{\text{Cov}({#1})}
\newcommand{\dirgr}{\text{{\bf DirGr}}}
\newcommand{\bdirgr}{\text{{\bf DirGr}}_b}
\newcommand{\coeff}{\text{Coeff}}
\newcommand{\pg}{P(\Gamma)}
\newcommand{\hh}{H\!H}
\newcommand{\AAA}{A\!\!-\!\! A}
\newcommand{\natr}{\stackrel{\bullet}{\longrightarrow}}
\newcommand{\grmod}{\text{{\bf GrMod}}}
\newcommand{\rmod}{\text{{\bf Mod}}}

\newcommand{\ralg}{\text{{\bf Alg}}}
\newcommand{\posets}{\text{{\bf Posets}}}
\title{The homology of digraphs as a generalisation of Hochschild homology}

\author{Paul Turner\thanks{The first author was partially supported by SNF project no. 200020-121506/1}  and Emmanuel Wagner  }

\institute{{\sc Paul Turner:} D\'epartement de math\'ematiques, 
Universit\'e de Fribourg, CH-1700 Fribourg, Switzerland.
\email{prt.maths@gmail.com}.
\hspace{1em}{\sc Emmanuel Wagner:} 
Institut de Math\'ematiques de Bourgogne, Universit\'e de Bourgogne, UMR 5584 du CNRS, BP47870, 21078 Dijon Cedex, France. \email{emmanuel.wagner@u-bourgogne.fr}.  
}

\titlerunning{Homology of digraphs}
\authorrunning{Paul Turner and Emmanuel Wagner}

\begin{document}

\maketitle

%%%%%%%%%%%%%%%%%%%%%%%%%%%%%%%%%%%%%%%%%%%%%%%%%%%%%%%%%%%%%%%%%%%%%%%%%

\begin{abstract}
J. Przytycki has established a connection between the Hochschild homology of an algebra $A$ and the chromatic graph homology of a polygon graph with coefficients in $A$. In general the chromatic graph homology is not defined in the case where the coefficient ring is a non-commutative algebra. In this paper we define a new homology theory for directed graphs which takes coefficients in an arbitrary $\AAA$ bimodule, for $A$ possibly non-commutative, which on polygons agrees with Hochschild homology through a range of dimensions.
\end{abstract}

%%%%%%%%%%%%%%%%%%%%%%%%%%%%%%%%%%%%%%%%%%%%%%%%%%%%%%%%%%%%%%%%%%%%%%%%%

\section*{Introduction}
This paper addresses a question of J. Przytycki. When defining the
Hochschild homology of an algebra with coefficients in a bimodule, the differential displays a
certain cyclical feature which makes it sometimes convenient to write the
tensor factors of $n$-chain generators, not linearly, but instead as
the vertices of an $(n+1)$-sided polygon. Przytycki
has a beautiful interpretation of this appearance of $n$-gons in terms of the 
chromatic homology of graphs developed by  of Helme-Guizon and Rong (\cite{HeGuRo}): a variant of this theory constructed using an algebra $A$ and an $\AAA$-bimodule $M$ applied to the $n$-gon  
is the Hochschild homology of $A$ with coefficients in $M$ through a range of dimensions
increasing with $n$. In general chromatic homology of a graph is only defined when the algebra is 
commutative, but in the case of the $n$-gon or a line graph, the
non-commutative case also makes sense. In speculating about possible generalisations
of Hochschild homology in \cite{Pr} Przytycki writes that he believes
``graph homology is the proper generalization of Hochschild homology:
from a polygon to any graph''. The remaining problem being, however,
that one does not know in general how to define chromatic graph homology
involving non-commutative algebras. To paraphrase, Przytycki's question
is: 
\begin{quotation}
{\em Given a (possibly non-commutative) algebra $A$ and an $\AAA$ bimodule $M$, can one construct
a functor from some category of graphs to graded modules such that on $n$-gons this functor agrees with  Hochschild homology
through a range of dimensions ?}
\end{quotation}

In this paper we consider finite based directed graphs, that is
digraphs whose vertex set is finite and for which there is a
distinguished vertex, the {\em base vertex}. Given a triple $(\Gamma,
A, M)$ consisting of a based digraph $\Gamma$, an $R$-algebra $A$ and
an $\AAA$ bimodule $M$ our main purpose is to define homology groups
$\cH_*(\Gamma, A, M)$ with nice functorial properties. 
There are two key ideas necessary to embrace
non-commutative algebras: firstly one must give up on having ``cube''
(as seen in the construction of chromatic graph homology) and secondly one
needs directed edges in order to be able to multiply non-commuting
elements (the head and tail providing information on which element
comes first).  The base vertex is necessitated by the appearance of the
bimodule $M$. In order to obtain a functor it is essential that we use the approach initiated in \cite{EvTu} relating Khovanov-type homology to the homology of posets with coefficients in a presheaf.

The construction goes roughly as follows. We
replace a digraph $\Gamma$ by its poset of directed multipaths $\pg$
and then define the homology groups $H_*(\Gamma, \cF)$ as the homology of the poset $\pg$ with
coefficients in an arbitrary coefficient system (pre-sheaf) $\cF$. Armed with these generalities we 
construct from a pair $(A,M)$ where $A$ is an algebra and $M$ an $\AAA$ bimodule, a particular coefficient system
$\cF_{A,M}$. This coefficient system generalises the
construction of the chromatic homology of graphs. The homology groups we are after are then defined by
$$
\cH_*(\Gamma, A,M) = H_*(\Gamma, \cF_{A,M}).
$$

%
%
%n
%$R$-algebra $A$ let $\bimod$ denote the category of $\AAA$-bidmodules.
%We define $\cH_*(M;\Gamma)$, the {\em homology of the $\AAA$-bimodule
%$M$ with coefficients in the based digraph $\Gamma$}. The definition
%allows for the possibility that $A$ is non-commutative and 

These homology groups satisfy some nice properties and give one
possible answer to Przytycki's question. Let $\bdirgr$ denote the
category of finite based digraphs with base vertex preserving
inclusions of digraphs as morphisms and let $\grmod$ denote the
category of $\bZ$-graded $R$-modules. We then have:

\vspace{5mm}

\noindent
{\bf Theorem \ref{thm:main}} {\em Let $A$ be an $R$-algebra  and $M$ an $\AAA$ bimodule. Then
$$
\cH_*(- ,A,M) \colon \bdirgr \ra \grmod 
$$
is a functor with the property that if  $\Gamma$ is a consistently directed $n$-gon then for $0\leq i \leq n-2$ 
$$
\cH_i(\Gamma, A, M ) \cong \hh_i (A;M)
$$
where on the right hand side we have the Hochschild homology of the algebra $A$ with coefficients in the bimodule $M$.
}

\vspace{5mm}

An interesting special case arises when $M=A$. In this case unbased digraphs suffice and we write
$\cH_*(\Gamma, A) = \cH_*(\Gamma, A,A)$. This is functorial in both variables:

\vspace{5mm}

\noindent
{\bf Theorem \ref{thm:restricted}}{\em  
$
\;\; \cH_*(- , -) \colon \dirgr \times \ralg  \ra \grmod 
$
is a  bifunctor.  
}
 \vspace{5mm}

By the previous theorem this bifunctor has the property  if  $\Gamma$ is a consistently directed $n$-gon then  for $0\leq i \leq n-2$,
$$
\cH_i(\Gamma ,A ) \cong \hh_i (A).
$$

%The homology of $M$ with
%coefficients in $\Gamma$ is then defined to be the homology of the
%path poset of $\Gamma$ with coefficients in $\cF_M$.

\section{The homology of directed graphs}

In this section we will be dealing with the category $\dirgr$ whose
objects are finite directed graphs (unbased) and whose morphisms are
inclusions of directed graphs. Our interest will be to construct
homology of directed graphs for an arbitrary coefficient system.
%To be precise about
%morphisms in the category $\dirgr$, we say that $f\colon \Gamma' \ra \Gamma$ is a morphisms if and only if $Vert(\Gamma)\subset
%Vert(\Gamma')$ and $Edge(\Gamma')\subset Edge(\Gamma).$ 
Note that a
directed graph $\Gamma$ comes with {\em tail} and {\em head }
functions 
$$ t,h\colon Edge(\Gamma) \ra Vert(\Gamma) $$ 
taking a
directed edge $e$ to its tail and its head respectively.

We will assume familiarity with the basics of posets. As usual we will
denote a partial ordering by $\leq$ with $x<y$ meaning $x\leq y$ and
$x\neq y$. We recall that an element $y$ is said to {\em cover}
another element $x$ if $x<y$ and there is no $z$ such that $x<z<y$. In
such a circumstance we write $x\prec y$.  The {\em Hasse diagram} of a
poset $P$ is the directed graph with one vertex for each element of
$P$ and an oriented arc from $x$ to $y$ if and only if $x\prec y$. The
{\em Boolean lattice} on a set is the poset of subsets partially
ordered by inclusion and its Hasse diagram is a hypercube. A poset may
be regarded as a category with one object for each element and a
unique morphism from $x$ to $y$ whenever $x\leq y$. Such morphisms
compose in the obvious way.

Let $\Gamma$ be a finite digraph and let $\bB (\Gamma)$ be the
Boolean lattice on its edge set $Edge(\Gamma)$.

\begin{definition}
A {\em simple path} in $\Gamma$ is a sequence of edges $e_1, \ldots
, e_n$ such that $h(e_i) = t(e_{i+1})$ and no vertex is encountered
twice. A {\em multipath} in $\Gamma$ is a collection of disjoint
simple paths.
\end{definition}

Let $P(\Gamma)$ be the subposet of $\bB(\Gamma)$ consisting of
multipaths in $\Gamma$. We will refer to it as the {\em path poset} of
$\Gamma$. By convention
there is one empty path $\emptyset$ and this is a (global) minimal
element for $\pg$ which we will denote by 0. There may be several
(local) maxima. As mentioned in  the previous paragraph we may choose to  view $\pg$ as a
category with a unique morphism between any two related elements. 

\begin{example}
In Figure \ref{fig:poset} we see a digraph $\Gamma$ along with an illustration of its path poset. 
\end{example}

\begin{figure}
\begin{center}
\includegraphics[angle=-90, scale = 0.5]{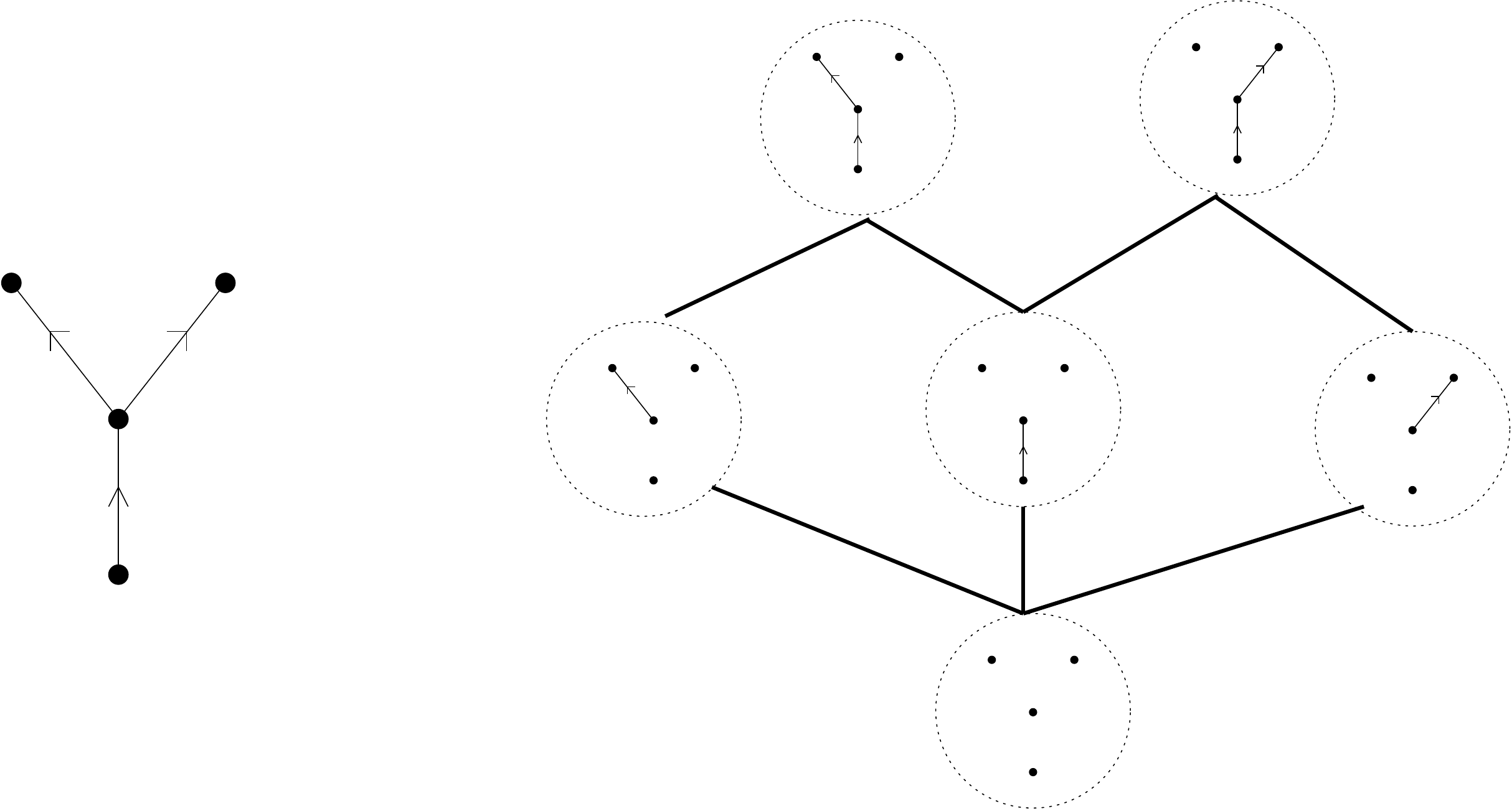}
\end{center}
\caption{A digraph and its path poset}\label{fig:poset}
\end{figure}

The assignment of a digraph to its path poset gives a covariant
functor $$ P(-) \colon \dirgr \ra \posets .  $$ 
Given an inclusion
$f\colon \Gamma^\prime \ra \Gamma$ we write $\tilde {f}\colon
P(\Gamma^\prime) \ra P(\Gamma)$ for $P(f)$. Note that this is an
injective map  of posets with the property that if $x\prec y$ in
$P(\Gamma^\prime)$ then $\tilde {f}(x)\prec \tilde {f}(y)$ in $P(\Gamma)$.

We note for later use the following lemma concerning such path posets.

\begin{lemma}\label{lem:diamond}$~$\\\vspace*{-5mm}
\begin{enumerate}
\item[(i)] For $x \in \pg$, the interval $[0,x]=\{y\in \pg \mid y\leq x\}$  is a Boolean lattice of rank $\gr x$.
\item[(ii)] If $x\prec y \prec z$ in $\pg$ then there exists a unique $y^\prime\neq y$ such that $x\prec y^\prime \prec z$.
\end{enumerate}
\end{lemma}

\begin{proof}
For (i) we simply note that given a multipath then each subset of its
edges is again a multipath. For (ii), the stated property is true for
Boolean lattices and so the result follows by considering the interval
$[0,z]$. \qed
\end{proof}

In general the homology of posets only becomes interesting (especially in the presence of a global minimum) if one allows local systems of coefficients.
Let $R$ be a commutative ring and let $\rmod$ be the category of $R$-modules. 

\begin{definition}
A {\em coefficient system} for a digraph $\Gamma$ consists of a covariant functor $\pg \ra \rmod$. These form the objects of a category $\coeff (\Gamma)$ whose morphisms are natural transformations of functors.
\end{definition}

One can take the homology of any small category with coefficients in a
functor and we now recall this construction restricted to our
particular setting i.e. where the category is a path category of a
digraph $\Gamma$ and the functor $\cF\colon P(\Gamma) \ra \rmod$ is a
coefficient system.  We will define a chain complex $\cS_*(\Gamma,
\cF)$ whose homology, by definition, is the homology of $\Gamma$
with coefficients in $\cF$.

We set
$$
\cS_k(\Gamma ; \cF) = \bigoplus_{x_0x_1\ldots x_k} \cF(x_0)
$$
where the sum is over all sequences $x_0 \leq x_1 \leq \cdots \leq x_k$ in $\pg$ of length $k+1$. A typical element is therefore a sum of elements of the form $\lambda x_0x_1\ldots x_k$ where $\lambda\in \cF(x_0)$. To turn this into a complex we define $d\colon \cS_k(\Gamma ; \cF) \ra \cS_{k-1}(\Gamma ; \cF) $ by 
$$
d(\lambda x_0x_1\ldots x_k) = \cF(x_0\leq x_1)(\lambda)x_1\ldots x_k 
+ \sum_{i=1}^k (-1)^i\lambda x_0 \ldots \hat{x_i}\ldots x_k.
$$ 
It is a standard fact (easily checked) that $d^2=0$ and so
$(\cS_*( \Gamma ; \cF), d)$ is a chain complex. We are now free to take
homology and we define the {\em homology of the directed graph}
$\Gamma$ {\em with coefficients in } $\cF$ to be the graded $R$-module
$$
H_*(\Gamma ;\cF) = H(\cS_*(\Gamma ; \cF), d).
$$

Homology has nice functorial properties as we see in the next proposition.

\begin{proposition}\label{prop:nat}$~$\\\vspace*{-5mm}
\begin{enumerate}
\item[(i)] Let $\Gamma$ be a finite directed graph. Then
$$
H_*(\Gamma; -)\colon \coeff (\Gamma)\ra \rmod
$$
is a covariant functor.
\item [(ii)] Let $f\colon \Gamma^\prime \ra \Gamma$ be an inclusion and let $\cF\colon P(\Gamma) \ra \rmod$ be a coefficient system for $\Gamma$. 
%Let $\cF^\prime$ be the induced coefficient system for $\Gamma^\prime$, i.e. $\cF^\prime = \cF \circ f$. 
Then there is an induced homomorphism
$$ f_* \colon H_*(\Gamma^\prime ; \cF \circ \tilde {f}) \ra H_*(\Gamma ; \cF). $$
Such induced homomorphisms are well behaved under
composition of inclusions.
\end{enumerate}
\end{proposition}
 
\begin{proof}
(i) Let $\cF$ and $\cG$ be coefficient systems and let $\tau$ be a
morphism from $\cF$ to $\cG$. That is, $\cF$ and $\cG$ are functors
$P(\Gamma) \ra \rmod$ and $\tau$ is a natural transformation $\cF \natr \cG$ consisting
of a map $\tau_x\colon \cF(x) \ra \cG(x)$ for each $x\in \pg$
satisfying the usual naturality requirements. We now define a
homomorphism
$$
\tau^\prime \colon \cS_k(\Gamma ; \cF) \ra \cS_k(\Gamma ; \cG)
$$
by setting
$$
\tau^\prime (\lambda x_0 \ldots x_k) = \tau_{x_0}(\lambda) x_0 \ldots x_k.
$$
The naturality of $\tau$ guarantees that this is a chain map and thus induces
$$
\tau_*\colon H_*(\Gamma ; \cF) \ra H_*(\Gamma ; \cG)
$$
as required. 

Furthermore, given another natural transformation $\sigma\colon \cG \natr \cK$ one has $(\sigma\tau)^\prime = \sigma^\prime \circ \tau^\prime$ from which it follows that $(\sigma\tau)_* = \sigma_* \circ \tau_*$.

(ii) Recalling that $\tilde{f}$ is the induced map on path posets, there is a homomorphism
$$
f^\prime\colon \cS_k(\Gamma^\prime ; \cF \circ \tilde {f}) \ra \cS_k(\Gamma ; \cF)
$$
defined by
$$
f^\prime(\lambda x_0 \ldots x_k) = \lambda \tilde {f} (x_0) \ldots \tilde {f}(x_k).
$$
This is a chain map since, by definition, $\cF \circ \tilde {f}(x\leq y) = \cF (\tilde {f}(x) \leq \tilde{f}(y))$. In homology this defines $f_*$. 

Given $g\colon \Gamma^{\prime \prime}\ra \Gamma^\prime$ we have $\widetilde{fg}= \tilde{f} \circ \tilde{g}$ from which it follows immediately that $(fg)_*=f_*\circ g_*$.
\qed
\end{proof}

Similar calculations to those in the above proof show that the maps $f_*$ are natural with respect to morphisms of coefficient systems. Spelt out more clearly this means the following.
Let $f\colon \Gamma^\prime \ra \Gamma$ be an inclusion of finite
 digraphs and let $\cF_1, \cF_2 \colon P(\Gamma) \ra \rmod$ be
 coefficient systems. Given a natural transformation $\tau\colon \cF_1\natr \cF_2$, define a natural transformation $\tilde{\tau}$ from
 $\cF_1\circ \tilde{f}$ to $\cF_2\circ \tilde{f}$ by $ \tilde{\tau}_x
 = \tau_{\tilde{f}(x)}$. Under such circumstances the following diagram commutes.

\[\xymatrix{
H_*(\Gamma^\prime ; \cF_1 \circ \tilde {f}) \ar[d]_{\tilde{\tau}_*}\ar[rr]^{f_{1*}} & & H_*(\Gamma ; \cF_1) \ar[d]^{\tau_*}\\
H_*(\Gamma^\prime ; \cF_2 \circ \tilde {f})  \ar[rr]^{f_{2*}}& & H_*(\Gamma ; \cF_2)
}\]

\section{The homology groups $\cH_*(\Gamma, A, M)$}

The category $\bdirgr$ has as objects finite digraphs that are
equipped with a preferred {\em base vertex}. Morphisms are inclusions
that take base vertex to base vertex. There is a forgetful functor
$\bdirgr \ra \dirgr$ and we can take homology by first applying this
functor and then proceeding as in the previous section. 

Our task in this section is to construct a coefficient system $\cF_{A,M} \colon
\pg \ra \rmod$, given a based digraph $\Gamma$, a (possibly non-commutative) unital $R$-algebra $A$ and  
an $\AAA$ bimodule $M$. Once achieved the main definition of the paper will be
$$
\cH_*(\Gamma, A,M) = H_*(\Gamma, \cF_{A,M}).
$$ 

We will take the tensor product of modules over
unordered sets so we recall here what this means. Let $S$ be a finite
set and suppose we have a family of $R$-modules indexed by $S$, that
is for each $\alpha\in S$ we have an $R$-module $M_\alpha$. The {\em unordered
tensor product} of this family, denoted 
$$
\bigotimes_{\alpha\in S} M_\alpha
$$ 
is formed by considering all possible orderings of the set $S$,
taking the direct sum of the ordered tensor product for each and then
identifying these via the obvious canonical isomorphisms induced from
permutations.

%By applying the homology of the previous section to
%$\Gamma$ with coefficients in $\cF_M$ we obtain what we will call the
%homology of $M$ with coefficients in the based digraph $\Gamma$.

We now proceed with the construction of $\cF_{A,M}\colon \pg \ra \rmod
$. For $x\in \pg$ let $\Gamma_x$ be the directed graph with the same
vertex set as $\Gamma$ and with edge set consisting of the edges in
the multipath $x$ (along with their directions). We will write
$\pi_0(\Gamma_x)$ for the set of connected components of $\Gamma_x$.
%This size of this set will be
%denoted $N_x$. 
We note that if $x\prec y$ then $\Gamma_y$ contains all the edges of
$\Gamma_x$ with one addition. Since all paths are simple this
additional edge clearly joins two separate components of
$\Gamma_x$. There is evidently a canonical identification of the
components of $\Gamma_x$ and $\Gamma_y$ not involved in this fusion.

For $x\in \pg$ consider the following family of $R$-modules $\{M_\alpha \}$
indexed by the set $\pi_0(\Gamma_x)$. If the component indexed by
$\alpha$ contains the base vertex then $M_\alpha=M$ otherwise
$M_\alpha=A$. Now we define $\cF_{A,M}(x)$ to be the unordered tensor product
$$
\cF_{A,M} (x) = \bigotimes_{\alpha \in \pi_0(\Gamma_x)} M_\alpha.
$$ 

To define the homomorphisms $\cF_{A,M}(x\leq y)$ we first consider
what happens in the case $x\prec y$. Here $\Gamma_y$ consists of
$\Gamma_x$ with an additional edge $e$ and as noted above two distinct
components in $\Gamma_x$ become one in $\Gamma_y$. The idea is to
define a homomorphisms using the canonical identification between
components away from those that fuse, and multiplication or the
actions of $A$ on $M$ for those that fuse. The key point is that the
order of multiplication is determined by the head and tail of $e$.

For temporary purposes let $I$ be an ordering of $\pi_0(\Gamma_x)$ and
$J$ be an ordering of $\pi_0(\Gamma_y)$. With respect to these
orderings suppose the two components of $\Gamma_x$ that fuse are
indexed by $i$ and $i^\prime$ and the new fused component in $\Gamma_y$ is
indexed by $j$. We now define a homomorphism (here we use the ordered
tensor product)

$$ 
\mu\colon M_{i} \otimes M_{i^\prime} \longrightarrow M_{j} 
$$ 
by 
$$
\mu(a\otimes b) = \begin{cases} 
	ab & \text{ if $i$ indexes the component containing $t(e)$}\\ 
	ba & \text{ if $i$ indexes the component containing $h(e)$} 
		\end{cases} 
$$ 
Here the expression $ab$ has several
possible meanings: if the base vertex is not involved in the fusion of
components then it means the multiplication in the algebra $A$; if
$M_i=M$ (i.e. $i$ indexes the component containing the base vertex)
then $M_j=A$ and $M_k=M$ and $ab$ means the right action of $A$ on
$M$; if $M_j=M$ (i.e. $j$ indexes the component containing the base
vertex) then $M_i=A$ and $M_k=M$ and $ab$ means the left action of $A$
on $M$. One similarly interprets $ba$. By combining this map with the
canonical permutation identification on the remaining tensor factors
this gives a homomorphism of ordered tensor products $\bigotimes_I M_i
\ra \bigotimes_J M_j$.

\begin{lemma}
The above defines a homomorphism 
$$
\cF_{A,M}(x\prec y) \colon \cF_{A,M}(x) \ra \cF_{A,M}(y).
%\bigotimes_{\pi_0(\Gamma_x)}A \ra \bigotimes_{\pi_0(\Gamma_y)} A .
$$
\end{lemma}
\begin{proof}
Let $\sigma$ be a permutation taking an ordering $I$ to another $I^\prime$. In the above construction the maps $\mu$ depend on the tensor factors corresponding to $t(e)$ and $h(e)$ not on the factors position in any ordering. It follows that there is a commutative diagram
$$
\xymatrix{
\bigotimes_I M_i \ar[r]^\mu \ar[d]_{\sigma} & \bigotimes_J M_j\\
\bigotimes_{I^\prime} M_{i^\prime} \ar[ru]^\mu
}
$$
and so the maps $\mu$ are compatible with the symmetric group action.
\qed
\end{proof}

\begin{lemma}\label{lem:seqlengthtwo}
If $x\prec y \prec z$ and $x\prec y^\prime \prec z$ then
$$
\cF_{A,M}( y \prec z) \circ  \cF_{A,M}( x \prec y) = \cF_{A,M}( y^\prime \prec z) \circ \cF_{A,M}( x \prec y^\prime).
$$
\end{lemma}
\begin{proof}
By Lemma \ref{lem:diamond} (ii) we know that $y$ and $y^\prime$ are
the only two elements lying between $x$ and $z$ in this way. Suppose
$\Gamma_y=\Gamma_x \cup \{e\}$ and $\Gamma_{y^\prime}=\Gamma_x \cup \{e^\prime\}$. If $e$ and $e^\prime$ are not both contained in a single simple path of $\Gamma_z$ then the result is clear. If they are both contained in the same simple path of $\Gamma_z$ then without loss of generality we may suppose that $e$ comes before $e^\prime$. Now choose an ordering on the components of $\Gamma_x$ so that the simple path before $e$ is labelled 1, the simple path between $e$ and $e^\prime$ is labelled 2 and the simple path after $e^\prime$ is labelled 3 (see Figure \ref{fig:paths}). 

\begin{figure}
\begin{center}
\includegraphics[angle=-90, scale = 0.4]{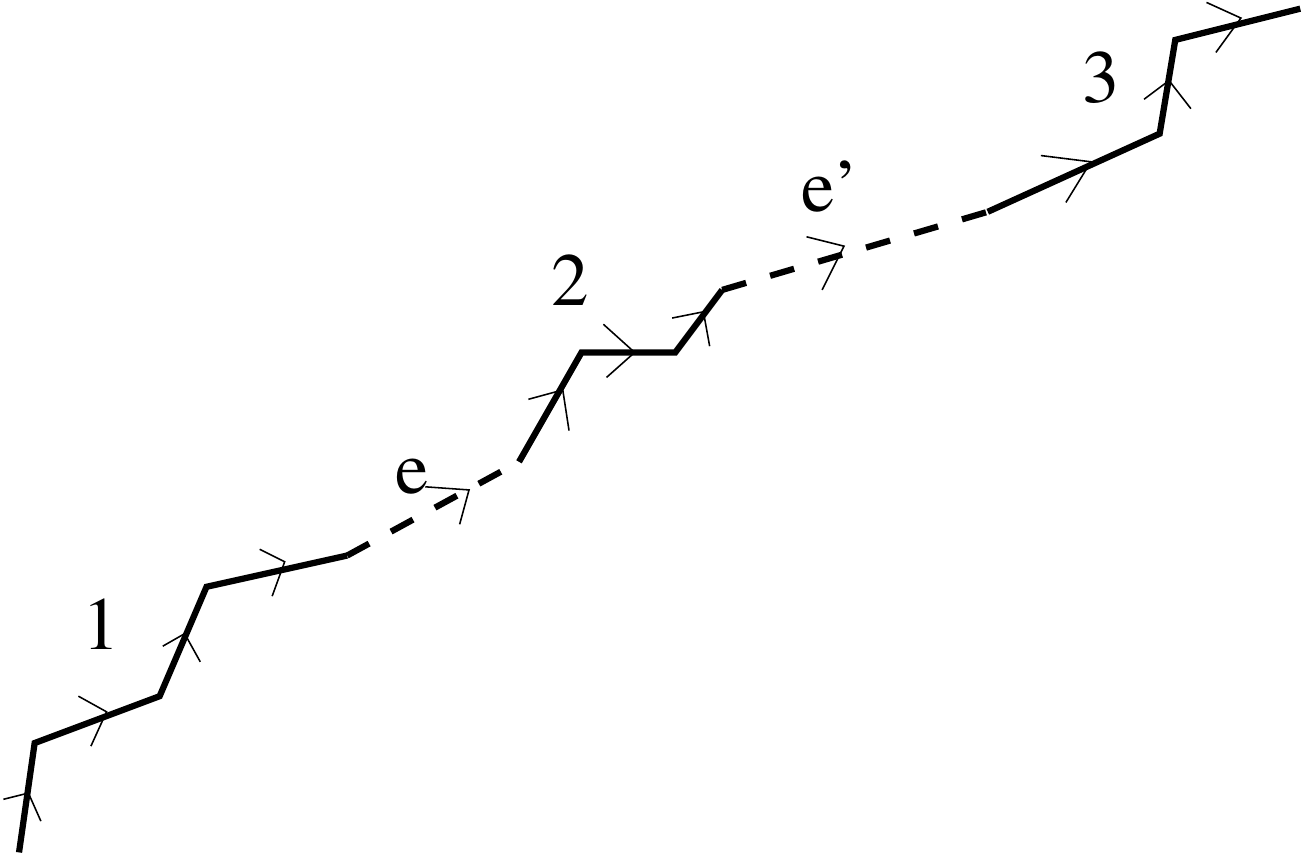}
\end{center}
\caption{Path ordering}\label{fig:paths}
\end{figure}

Now suppose we have ordering of the components of $\Gamma_y$ 
%and $\Gamma{y^\prime}$ 
such that $1$ indexes the component of $\Gamma_y$ containing $e$ and 2
indexes the component at the head of $e^\prime$. Similarly, suppose we
have ordering of the components of $\Gamma_{y^\prime}$ such that $1$
indexes the component at the tail of $e$ and 2 indexes the component
containing $e^\prime$. In $\Gamma_z$ the simple path containing $e$
and $e^\prime$ is indexed by $1$. Then by the associativity of $\mu$
the following diagram commutes
$$
\xymatrix{
M_1\otimes M_2 \otimes M_3 \ar[r]^{\mu \otimes 1} \ar[d]^{1 \otimes \mu }&
M_1\otimes M_2  \ar[d]^{\mu } \\
M_1\otimes M_2  \ar[r]^{\mu } & M_1
}
$$ 
from which the result easily follows.
\qed
\end{proof}

We now extend this to define a map $\cF_{A,M}(x\leq y)\colon \cF_{A,M}(x) \ra \cF_{A,M}(y)$ for any $x\leq y$. Pick a sequence $x\prec x_1 \prec \cdots \prec x_l \prec y$ and set
$$
\cF_{A,M}(x\leq y) = \cF_{A,M}( x_l \prec y) \circ \cdots \circ \cF_{A,M}( x \prec x_1).
$$
Courtesy of Lemma \ref{lem:seqlengthtwo} we immediately see that this does not depend on the particular choice of sequence. We have thus shown that 

\begin{proposition}
$\cF_{A,M}\colon \pg \ra \rmod $ as defined above is a covariant functor,
i.e. $\cF_{A,M}$ is a coefficient system for $\Gamma$.
\end{proposition}

We finally arrive at the principal definition of this section. 

\begin{definition}
Let $A$ be a unital $R$-algebra, $M$ an $\AAA$ bimodule and $\Gamma$ a
finite based digraph. Using the coefficient system $\cF_{A,M}$ above we
define
 $$
\cH_*(\Gamma, A,M) = H_*(\Gamma, \cF_{A,M}).
$$ 
\end{definition}

The following theorem provides one possible answer to Przytycki's question.

\begin{theorem}\label{thm:main}
Let $A$ be an $R$-algebra  and $M$ an $\AAA$ bimodule. Then
$$
\cH_*(- ,A,M) \colon \bdirgr \ra \grmod 
$$
is a functor with the property that if  $\Gamma$ is a consistently directed $n$-gon then for $0\leq i \leq n-2$ 
$$
\cH_i(\Gamma, A, M ) \cong \hh_i (A;M)
$$
where on the right hand side we have the Hochschild homology of the algebra $A$ with coefficients in the bimodule $M$.
\end{theorem}

Before proving this Theorem let us recall Przytycki's result relating
the chromatic homology of graphs to Hochschild homology
\cite{Pr}. We will state the results using homological grading
conventions. Firstly recall that (homologically graded) chromatic
homology of a graph $G$ is defined as follows. Let $A$ be a {\em
commutative} $R$-algebra. Let $\bB$ be the Boolean lattice (the
``cube'' as it is usually referred to) on the edges of
$G$. An element of $\bB$ is a subgraph of $G$ with the same
vertex set as $V$ and will typically contain some isolated
vertices. To each such subgraph $x$ associate the module $\cF(x)$
being a tensor product of copies of $A$, one for each connected
component. To each edge $\zeta$ of the cube (covering relation in
the Boolean lattice) associate a map $d_\zeta$ being given by the algebra
multiplication if two connected components fuse, or the identity map otherwise.
 Letting $N$ be the number of edges in $G$, one defines, for $i=0,1, \ldots , N$
$$
\cC_i (\Gamma) = \bigoplus_{x\in \pg, \gr x = N-i} \cF (x).  
$$
A differential  $d\colon \cC_i( \Gamma) \ra \cC_{i-1}( \Gamma)$ can be defined for $a\in \cF(x)$ by 
$$
d(a) = \bigoplus_{\zeta} \epsilon(\zeta) d_\zeta (a),
$$
where $\epsilon(\zeta)=\pm 1$. This gives a complex, whose homology is the chromatic graph homology of $G$ using the algebra $A$. All this is well documented elsewhere (see \cite{HeGuRo,HeGuPrRo} for details).

Przytycki extends the above in the following way (see \cite{Pr} for details). Suppose $M$ is an
$\AAA$-bimodule (where as above $A$ is commutative) and suppose $v_1$ is a chosen 
 base vertex of $G$. Modify the above construction by replacing $A$ by $M$
whenever associating a module to a component containing
$v_1$. Moreover, in the definition of the differential, partial
derivatives between subgraphs having the same number of components are
set to zero. Denote the result $\hat{H}_*^{A,M}(G)$.

If we take $G=P_n$, an $n$-sided polygon, then taking coherent
directions on each edge (so that the whole polygon is oriented
clockwise or anti-clockwise) then Przytycki argues the above may be
extended to the case where $A$ is non-commutative. (From our point of
view, the existence of this global orientation, means that each
subgraph is again consistently directed and is thus a multipath in our
set up).

Przytycki's main result (with homological grading conventions) is:

\vspace{3mm}

\noindent
{\bf Theorem}(Przytycki)
$$
\hat{H}_i^{A,M}(P_n) \cong \hh_{i-1} (A;M)\;\;\;\;\;\;\;\; \text{for $1\leq i \leq n-1$}.
$$

Thus, modified chromatic homology of polygons agrees with Hochschild homology through a range of dimensions. We are now ready to prove Theorem \ref{thm:main} above.
 
%Let $P_n$ be an
%$n$-sided polygon and $\bB$ the Boolean lattice on its edges - the
%so-called ``cube''. The elements of this lattice consist of subgraphs
%of $P_n$ with the same vertex set as $P_n$. We suppose further that
%one particular vertex of $P_n$ has been chosen as a base vertex. To a
%subgraph $H$ associate a module $\Phi(H)$ as follows. To a connected
%component of $H$ associated either $M$ if the component contains the
%base vertex or $A$ otherwise, then set $\Phi(H)$ to be the tensor
%product of these factors.

\begin{proof} (of Theorem \ref{thm:main})
Let $f\colon \Gamma^\prime \ra \Gamma$ be a (basepoint preserving) inclusion. We wish to define a chain map
$$
f^\prime\colon \cS_*(\Gamma^\prime, \cF^\prime_{A,M})\ra \cS_*(\Gamma, \cF_{A,M})
$$
Here we are writing $\cF^\prime_{A,M}$ for the coefficient system constructed above for the graph $\Gamma^\prime$.

Firstly, we construct a morphism of coefficient systems $\tau\colon \cF^\prime_{A,M}\natr \cF_{A,M}\circ \tilde{f}$. For $x\in
P(\Gamma)$ the graph $\Gamma_{\tilde{f}(x)}$ is isomorphic
to the graph $\Gamma^\prime_x \cup W $ where $W$  consists of the
vertices in $\Gamma$ which are not in (the image of) $\Gamma^\prime$. We can thus make the identification
$$
\cF_{A,M}(\tilde{f}(x)) \cong  \cF^\prime_{A,M}(x) \otimes \bigotimes_{W}A.
$$
Moreover, if $x\prec y$ in $\Gamma^\prime$  then these identifications make the following diagram commute.
\begin{equation*}
\xymatrix{\cF_{A,M}(\tilde{f}(x)) \;\;\;\;\;\; \cong  \ar[d]_{\cF_{A,M}\circ \tilde{f}(x\prec y)}& \cF^\prime_{A,M}(x) \otimes \bigotimes A \ar[d]^{\cF^\prime_{A,M}(x\prec y)} \\
\cF_{A,M}(\tilde{f}(y)) \;\;\;\;\;\; \cong  & \cF^\prime_{A,M}(y) \otimes \bigotimes A
}
\end{equation*}

We now define $\tau_x$ to be the composition
$$
\cF^\prime_{A,M}(x) \cong 
\cF^\prime_{A,M}(x) \otimes_{R} \bigotimes R \ra
\cF^\prime_{A,M}(x) \otimes_{R} \bigotimes A \cong
\cF_{A,M}(\tilde{f}(x))
$$
where the map shown is just the identity on $\cF^\prime_{A,M}(x) $ and the unit map of the algebra $R\ra A$ on the remaining tensor factors. 

For $x\prec y$ in $\Gamma^\prime$, the diagram above shows that
$$
\xymatrix{\cF^\prime_{A,M}(x)  \ar[d]_{\cF^\prime_{A,M}(x\prec y)}  \ar[r]^(.4){\tau_x} & \cF^\prime_{A,M}(x) \otimes \bigotimes A \ar[d]^{\cF^\prime_{A,M}(x\prec y)\otimes Id} \\
\cF^\prime_{A,M}(y)  \ar[r]_(.4){\tau_y} & \cF^\prime_{A,M}(y) \otimes \bigotimes A
}
$$
Now suppose $x\leq y$ in $\Gamma^\prime$. 
We may choose a sequence $x\prec x_1 \prec \cdots \prec x_m \prec y$ and noting that $\tilde{f}(x)\prec \tilde{f}(x_1) \prec \cdots \prec \tilde{f}(x_m) \prec \tilde{f}(y)$ we see $\tau$ is natural by repeated use of the  above diagram. Thus it is thus a morphism of coefficient systems as desired.

From Proposition \ref{prop:nat} (i) we now get a homomorphism
$$
\tau_* \colon H_*(\Gamma^\prime , \cF^\prime_{A,M}) \ra H_*(\Gamma^\prime , \cF_{A,M}\circ \tilde{f}).
$$

Invoking part (ii) of  Proposition \ref{prop:nat} we also have a homomorphism
$$
f_* \colon H_*(\Gamma^\prime , \cF_{A,M}\circ \tilde{f}) \ra H_*(\Gamma , \cF_{A,M}).
$$

The composition $f_*\circ \tau_*$ gives a homomorphism
$$
f_\bullet \colon \cH_*(\Gamma^\prime, A, M) \ra \cH_*(\Gamma, A, M)
$$
which is the map in homology induced by $f$.

It remains to show that if $g\colon \Gamma^{\prime\prime} \ra
\Gamma^\prime$ is an inclusion then $(fg)_\bullet = f_\bullet \circ
g_\bullet$. This amounts to showing that the top and bottom routes around the
following diagram are the same (where we have omitted the subscripts $A,M$ and are
being a little liberal in our multiple uses of the letter $\tau$).

{\small
$$
\xymatrix{
 & H_*(\Gamma^{\prime\prime}, \cF^{\prime}\circ \tilde{g}) \ar[r]^{g_*} \ar[rdd]_{\tilde{\tau}_*} &
H_*(\Gamma^{\prime}, \cF^{\prime}) \ar[r]^{\tau_*} &
H_*(\Gamma^{\prime}, \cF\circ \tilde{f}) \ar[rd]^{f_*}& \\
H_*(\Gamma^{\prime\prime}, \cF^{\prime\prime}) \ar[ru]^{\tau_*} \ar[rrd]_{\tau_*}& 
&&&
H_*(\Gamma, \cF)\\
&& H_*(\Gamma^{\prime\prime}, \cF\circ \widetilde{fg}) \ar[uur]_{g_*} \ar[urr]_{(fg)_*} &&
}
$$
}

The left-hand triangle commutes directly from the definition of the maps $\tau$, and the right-hand triangle commutes from the statement in Proposition \ref{prop:nat} (ii) that the induced maps behave well under composition. The middle square commutes by the comments immediately after the proof of Proposition \ref{prop:nat}.

In order to make the connection with Hochschild homology we combine
Przytycki's result with the work of Everitt and the first author
\cite{EvTu}. It is clear that if $\Gamma= P_n$ with consistent
directions, then $\pg$ is the Boolean lattice on the edges of $P_n$  minus its
maximum element.

Moreover the functor $\cF_{A,M}$ constructed on $\pg$ agrees with the
(implicit) functor used in the construction of $\hat{H}_{i+1}^{A,M}$ in this case. Since the
graph is a polygon, the only partial derivative in Przytycki's set-up
that are set to zero are those from corank $1$ elements to the maximal
element and this corresponds to the absence of the maximum element of
the Boolean lattice in $\pg$.

It now follows from the main result of \cite{EvTu} that the
homology of the path category of category $P_n$ with coefficients in $\cF_{A,M}$ is
isomorphic to $ \hat{H}_{i+1}^{A,M}(P_n)$ with a grading shift:
$$
\cH_i(P_n, A,M) \cong H_i(P_n , \cF_{A,m}) \cong \hat{H}_{i+1}^{A,M}(P_n).
$$
Combining this with Przytycki's result gives the desired isomorphism.
\qed
\end{proof}

When $M=A$ the base vertex become irrelevant and we may define for an
(unbased) digraph $\Gamma$ the homology groups $\cH_*(\Gamma, A) =
\cH_*(\Gamma, A,A)$, where on the right-hand side any base vertex for
$\Gamma$ will do. Letting $\ralg$ denote the category of $R$-algebras we have:

\begin{theorem}\label{thm:restricted}
$
\cH_*(- , -) \colon \dirgr \times \ralg  \ra \grmod 
$
is a  bifunctor.  
\end{theorem}

\begin{proof}
Functoriality in the first variable follows from the previous theorem. 
For the second variable, let $f\colon A \ra B$ be an algebra homomorphism. For $x\in \pg$ we have 
$$\cF_{A,A}(x) = \bigotimes A  \;\;\;\; \text{   and    } \;\;\;\; \cF_{B,B}(x) = \bigotimes B
$$ where the tensor product is over the same indexing set in both
 cases. We can therefore define a homomorphism $\sigma_x\colon
 \cF_{A,A}(x) \ra \cF_{B,B}(x)$ by $\sigma_x=\bigotimes f$. Since $f$
 is a homomorphism of algebras, the $\sigma_x$ define a natural
 transformation $\sigma \colon \cF_{A,A}
 \stackrel{\bullet}{\longrightarrow} \cF_{B,B}$. By Proposition
 \ref{prop:nat} (i) this induces a map $H_*(\Gamma, \cF_{A,A} ) \ra
 H_*(\Gamma, \cF_{B,B} )$ as required. Composition of algebra
 homomorphisms is easily seen to give a well defined composition of
 these induced maps.  \qed
\end{proof}

If one fixes the directed graph $\Gamma$ the above gives a functor
$$
\cH^{\Gamma}_*(-)\colon \ralg  \ra \grmod .
$$ 
One is tempted to call $\cH^{\Gamma}_*(A)$ the {\em homology of the
algebra $A$ with coefficients in the digraph $\Gamma$}. Such homology
theories of algebras are probably worthy of study in their own
right. We limit ourselves here to the observation that if $\gamma$ is
an oriented cycle of length $n$ in $\Gamma$ then there is an inclusion
of digraphs $\gamma \ra \Gamma$ which by functoriality and Theorem \ref{thm:main} gives a map
$$
\gamma_*\colon \hh_{i}(A) \ra \cH^{\Gamma}_*(A)
$$ 
for $i=0,1,\ldots , n-2$.

\end{document}